\documentclass[11pt]{amsart}

\usepackage[T2A]{fontenc}
\usepackage[utf8]{inputenc}
\usepackage[russian,english]{babel}

\usepackage{amssymb}

\usepackage{color}
\usepackage{hyperref}

\newtheorem{theorem}{Theorem}[section]
\newtheorem{lemma}[theorem]{Lemma}

\newtheorem{corollary}[theorem]{Corollary}

\numberwithin{equation}{section}

\newcommand{\bbR}{\mathbb{R}}

\newcommand{\bbN}{\mathbb{N}}

\newcommand{\bbQ}{\mathbb{Q}}

\newcommand{\bbZ}{\mathbb{Z}}

\newcommand{\cH}{\mathcal{H}}

\renewcommand{\Re}{\mathop{\mathrm{Re}}}

\newcommand{\SO}{\mathop{\mathrm{SO}}}

\newcommand{\al}{{\mathord{\alpha}}}
\newcommand{\be}{{\mathord{\beta}}}
\newcommand{\ga}{{\mathord{\gamma}}}
\newcommand{\Ga}{{\mathord{\Gamma}}}

\newcommand{\la}{{\mathord{\lambda}}}

\newcommand{\De}{{\mathord{\Delta}}}

\newcommand{\ep}{{\mathord{\varepsilon}}}
\newcommand{\eps}{{\mathord{\epsilon}}}
\newcommand{\ze}{{\mathord{\zeta}}}

\newcommand{\frR}{\mathord{\mathfrak{R}}}
\newcommand{\frS}{\mathord{\mathfrak{S}}}

\newcommand{\tsfT}{\hbox{\tiny\sf T}}

\let\td=\tilde

\begin{document}

\title[]
{Factorization of special harmonic polynomials of three variables}


\author{V. Gichev}
\address{Sobolev Institute of Mathematics\\
Omsk Branch\\
ul. Pevtsova, 13, \\ 644099, Omsk, Russia}
\email{gichev@ofim.oscsbras.ru}
\curraddr{}




\begin{abstract}
We consider harmonic polynomials of real variables $x,y,z$ that are eigenfunctions of the rotations 
about the axis $z$. They have the form $(x\pm yi)^{n}p(x,y,z)$, where $p$ is a rotation invariant polynomial. Let 
$\frR_{m}$ be the family of the polynomials $p$ of degree $m$ which are reducible over the rationals.  
We describe $\frR_{m}$ for $m\leq5$ and prove that $\frR_{6}$ and $\frR_{7}$ are finite. 
\end{abstract}

\maketitle


\section{Introduction}
The space $\cH$ of all complex valued harmonic polynomials on $\bbR^{3}$ admits a linear basis of the form 
\begin{eqnarray*}
f(x,y,z)=(x\pm yi)^{k}p(x,y,z), 
\end{eqnarray*}
where $p$ is a homogeneous polynomial which depends only on $x^{2}+y^{2}$ and $z$. 
The polynomials $f$ are harmonic eigenfunctions of the circle group $T_{o}$ of rotations about the axis $z$. 
The polynomials $p$ are $T_{o}$-invariant and, moreover, they are closely related to the associated Legendre 
functions $P_{n}^{m}$ as well as to the Laplace functions $Y_{n}^{m}$.  If $k=0$, then a change of variables and 
normalization 
converts $p$ into the Legendre polynomial $P_{n}$. In this paper, we consider the problem of reducibility 
over the field $\bbQ$ for the polynomials $p$ of small degrees. This is evidently connected with the problem of 
existence of common zeroes of a couple of harmonic polynomials and a description of their nontrivial factors. 

\subsection{Brief history} 
Algebraic and arithmetic properties of classical orthogonal polynomials are still not well understood. 
(Maybe, Tchebyshev's polynomials could be considered as an exception.)  In 1890,  
Stieltjes in a letter to Hermite (see \cite[Letter 275]{St}) formulated two questions: first, are there two distinct 
polynomials $P_{n}$ and $P_{m}$ that have a nontrivial common root?  and second, are $P_{n}(x)$ for even 
$n$ and $\frac{P_{n}(x)}{x}$ for odd $n$ irreducible?  The existence of a common root implies the reducibility 
of at least one of the polynomials. Hence the affirmative answer to the first question implies the negative answer 
to the second. However, the second question remains unanswered yet. In several special situations,  the 
irreducibility of $P_{n}$ was proved by Holt  \cite{Ho1912a}--\cite{Ho1913}, Ille \cite{Il1924}, 
McCoart \cite{McC61}, Melnikov \cite{Me56}, and  Wahab \cite{Wa52}, \cite{Wa60}. 
According to their results $P_{n}$ is irreducible for all $n\leq 500$ with a few 
 possible exceptions.  

In 1991, Armitage \cite{Ar91} proved that an entire harmonic function vanishes on the cone 
$x_{1}^{2}=\al^{2}(x_{1}^{2}+\cdots+x_{n}^{2})$ in $\bbR^{n}$, where $0<\al<1$,  if and only if $\al$ is a root of an 
ultraspherical polynomial or of some its derivative. 
Thus, we return to the delicate problem of common zeros of orthogonal polynomials again. It is worth 
mentioning that the Legendre 
polynomials cannot have quadratic factors (see \cite[Theorem~1.7 and Corollary~3.4]{Wa52}). 

The complete characterization of the quadratic divisors of harmonic polynomials on $\bbR^{n}$ in terms of 
polynomial solutions to some Fuchsian ordinary differential equations was obtained by Agranovsky and Krasnov 
in the paper \cite{AK00}.  In the earlier paper \cite{Ag99}, Agranovsky proved that a product of linear forms is 
harmonic if and only if the family of its zero hyperplanes forms a Coxeter system. The investigation was motivated by 
problems of mathematical physics and PDE, particularly the problems of injectivity of the Radon transform and 
stationary points of the wave equation. These papers also contain some open questions.  
Here is one of them (\cite[\S6, question 4]{Ag99}): does there exist a harmonic homogeneous polynomial of degree 
greater than 5 which is divisible by $x^{2}+y^{2}-2z^{2}$? The answer is still unknown. Such  polynomials of 
degree 5 exist: the real and imaginary parts of $(x+yi)^{2}z(x^{2}+y^{2}-2z^{2})$ are harmonic. 

The recent paper \cite{MWW19} by Mangoubi and Weller Weiser  is actually devoted to this problem. They formulate   
a conjecture: the space of harmonic functions on the unit ball $B\subset\bbR^{3}$  which vanish on 
the part of the cone $x^{2}+y^{2}-2z^{2}=0$ that lies in $B$ is finite dimensional (\cite[Conjecture~1]{MWW19}). It is 
weaker than the property stated in the above question. (The authors probably were not aware of it.)  
Anyway, their results support both. The methods are arithmetic  and algebraic. 

Very interesting papers \cite{LM15}, \cite{LM16} by Logunov and Malinnikova clarify the dependence of the 
harmonic functions on their zero sets. 
Example~4.3 in \cite{LM15} shows that the space of harmonic polynomials which vanish on the cone  
$3x_{0}^{2}=x_{1}^{2}+x_{2}^{2}+x_{3}^{2}$  in $\bbR^{4}$ is infinite dimensional as well as the space of 
harmonic polynomials divisible by $3x_{0}^{2}-x_{1}^{2}-x_{2}^{2}-x_{3}^{2}$. 
It is not clear if there exists a nontrivial round cone in $\bbR^{3}$ with this property. 
\subsection{Notation} The harmonic polynomial of the type 
\begin{eqnarray*}
(x+yi)^{k}p(x,y,z), 
\end{eqnarray*}
where $p$ is homogeneous of some degree and $T_{o}$-invariant, is unique up to a multiplicative constant. 
It is convenient to 
deal with complex valued polynomials while the divisibility problem is more interesting in the case of real ones. 
Without loss of generality we may assume $p$ real. Hence the real part of the above polynomial is the product of $p$ and $\Re (x+yi)^{k}$ and the same is true for the imaginary part. 

The condition that the coefficient at the highest power of $x$ in $p$ (equivalently, at the highest power of $y$ or 
at $x^{2}+y^{2}$) is equal to 1 uniquely defines $p$. To avoid cumbersome indices, we use the following 
notation for these polynomials: 
\begin{eqnarray*}
f_{n,d}(x,y,z)=(x+yi)^{n-\left[\frac{d}2\right]\,}p_{n,d}(x,y,z),
\end{eqnarray*}
where the brackets stand for the integer part of a number and $d=\deg p_{n,d}$.  Thus 
$\deg f_{n,d}$ is either $n+d$ or $n+d+1$ for $d$ even and $d$ odd,  respectively. 

In the sequel, ``reducibiulity'' means ``reducibility over $\bbQ$''. Set 
\begin{eqnarray*}
\frR_{d}=\{f_{n,d}:\,\De f_{n,d}=0~~\text{and}~~ p_{n,d}~~\text{is reducible}\}.
\end{eqnarray*}
Let $t=\frac{x^{2}+y^{2}}{z^{2}}$ and put  
\begin{eqnarray*}
\td p_{n,d}(t)=\frac{p_{n,d}(x,y,z)}{z^{d}}. 
\end{eqnarray*}
Then $f_{n,d}\in\frR_{d}$ if and only if the polynomial $\td p_{n,d}$ is reducible. If $d\leq7$, then 
$\deg\td p_{n,d}\leq3$. Hence for $d\leq7$ we have the equivalence 
\begin{eqnarray*}
f_{n,d}\in\frR_{d}~~~\Longleftrightarrow~~~\td p_{n,d}~~\text{has a rational root}.
\end{eqnarray*}
For short, we denote 
\begin{eqnarray*}
\ze=x+yi. 
\end{eqnarray*}
\subsection{Methods and results} 
The case $d\leq3$ is trivial: the polynomials $\ze^{n}$ and $\ze^{n}z$ are always harmonic, for $d=2,3$  
$\De f_{n,d}=0$ if and only if there is $n\in\bbN$ such that 
\begin{eqnarray*}
&f_{n,2}=\ze^{n-1}(\ze\bar\ze-2nz^{2}),\\
&f_{n,3}=\ze^{n-1}z\left(\ze\bar\ze-\frac{2}{3}nz^{2}\right),
\end{eqnarray*} 
respectively. Thus we assume $4\leq d\leq7$ in the sequel. 

There is one-to-one correspondence between the solutions $u,v\in\bbN\setminus\{1\}$ 
to the Pell equation $u^{2}-6v^{2}=1$ and  $\frR_{4}$. For $\frR_{5}$ there are three similar series of solutions 
to the Pell type equation $u^{2}-10v^{2}=9$. The methods are elementary: the polynomials $\td t_{d}$ are quadratic 
and the rationality of a root is equivalent to some Diophantine equation since $n$ is integer.

In the case $d=6$ we prove that  $n=Ku^{3}$ and $2n+1=Lv^{3}$, where $u,v\in\bbN$ and $K,L$ runs over 
a finite subset of $\bbN$.  Hence $u,v$ satisfy the Diophantine equation $Lv^{3}-2Ku^{3}=1$. 
If $d=7$, then a similar assertion holds for $n$, $2n+3$, and the equation $Lv^{3}-2Ku^{3}=3$. The sets 
$\frR_{6}$ and $\frR_{7}$ are finite because every Diophantine equation above 
has a finite number of solutions  due to Thue--Siegel--Roth theorem on Diophantine approximation. 

\subsection{Reminder on the Pell equation}\label{pellintro}  This is the Diophantine equation $u^{2}-Dv^{2}=1$, where $D\in\bbN$
is not a square. Let $Q$ be the quadratic form on the left, $\phi=(u,v)\in\bbZ^{2}$, 
\begin{eqnarray*}
&\frS=\{\phi\in\bbN\times\bbZ:\, Q(\phi)=1\},\\
&\frS^{+}=\frS\cap(\bbN\times\bbN).
\end{eqnarray*}
The solution $\phi_{1}=(u_{1},v_{1})\in\frS^{+}$ such that $u_{1}=\min\{u:\,\phi\in\frS^{+}\}$ 
is called fundamental. The subgroup $\Ga$ of $\SO(Q,\bbZ)$ generated by 
\begin{eqnarray*}
M=\left(\begin{array}{cc}u_{1}&Dv_{1}\\ v_{1}&u_{1}\end{array}\right)
\end{eqnarray*}
preserves the lattice $\bbZ^{2}$, the quadratic form $Q$, and consequently the set $\frS$.  
Moreover, it leaves invariant the branch $B$ of the hyperbola $Q=1$ which contains $\phi_{0}=(1,0)$. 
The arc of $B$ with endpoints $\phi_{0},\phi_{1}$ is a fundamental domain of $\Ga$ in $B$.  
Since it contains no solution except for the endpoints by definition of $\phi_{1}$, $\Ga$ is transitive on $\frS$ and, moreover, 
\begin{eqnarray*}
\frS^{+}=\{M^{k}\phi_{0}:\,k\in\bbN\}.
\end{eqnarray*}  
Put $\phi_{k}=M^{k}\phi_{0}=(u_{k},v_{k})$. Any linear combination $X_{k}$ of the sequences 
$u_{k},v_{k}$ is subject to the formula 
\begin{eqnarray*}
X_{k}=C_{1}\la_{1}^{k}+C_{2}\la_{2}^{k},
\end{eqnarray*}
where $\la_{1},\la_{2}$ are the eigenvalues of $M$ and  $C_{1},C_{2}$ are constant, 
and consequently satisfies the recurrence relation 
\begin{eqnarray*}
X_{k+1}-2u_{1}X_{k}+X_{k-1}=0.
\end{eqnarray*}

\subsection{Acknowledgements} 
I am grateful to Mark Agranovsky for making me aware of this circle of problems and for 
fruitful discussions on them and the relating questions. 


\section{Reducibility of $p_{n,4}$ and the Pell equation}

\subsection{The Diophantine equation} Let the polynomial 
\begin{eqnarray*}
f_{n,4}
=\ze^{n-2} (\ze\bar\ze - A_{n}z^2) (\ze\bar\ze - B_{n}z^2)
\end{eqnarray*}
of degree $n+2$ be harmonic. Computing $\De f_{n,4}$, we get the equalities 
\begin{eqnarray*}
\begin{cases}
A_{n}+B_{n}=4n,\\
A_{n}B_{n}=\frac43\,n(n-1). 
\end{cases}
\end{eqnarray*}
Therefore,  
\begin{eqnarray}\label{ABnm4}
\begin{cases}
A_{n}=2\Big(n-\sqrt{\frac{2n^{2}+n}{3}}\Big),\\
B_{n}=2\Big(n+\sqrt{\frac{2n^{2}+n}{3}}\Big). 
\end{cases}
\end{eqnarray} 
Hence $0<A_{n}\leq B_{n}$. Since $n\in\bbN$, either $m=\sqrt{\frac{2n^{2}+n}{3}}$ is integer 
or both $A_{n}$ and $B_{n}$  are irrational. Thus we have the Diophantine equation 
\begin{eqnarray}\label{eqnm}
2n^{2}+n=3m^{2}
\end{eqnarray}
with positive integer $n,m$. There are two evident solutions: $n=m=0$ and $n=m=1$. They correspond 
to the harmonic polynomials $\bar\ze^{2}$ and $\bar\ze(\ze\bar\ze-4z^{2})$ relating to the cases $d=0$
and $d=2$, respectively. We do not take them in account. Thus we assume $$n>1.$$  
Then the converse is also true: {\it for every solution $(n,m)\in(\bbN\setminus\{1\})\times\bbN$ 
to (\ref{eqnm}) we have $\deg f_{n,4}=n+2$ and $\De f_{n,4}=0$.} 

\subsection{Description of $\frR_{4}$} Let the sequence $\al_{k}$ be defined by 
the recurrence relation 
\begin{eqnarray}\label{recrel6}
X_{k+1}=10X_{k}-X_{k-1}
\end{eqnarray}
and the initial data $\al_{0}=\al_{1}=1$. Set 
\begin{eqnarray}\label{aknk4}
\begin{array}{rcl}
a_{k}&=&\frac12(\al_{k}-1),\\
n_{k}&=&\frac18(\al_{k}+\al_{k+1}-2). 
\end{array}
\end{eqnarray}
Note that $n_{k}=\frac14(a_{k}+a_{k+1})$. We shall prove that $a_{k}=A_{n_{k}}$ in the notation  
of (\ref{ABnm4}). These numbers are integer. There is a short table at the end of the section. 
\begin{theorem}
The family $\frR_{4}$ consists of the polynomials 
\begin{eqnarray*}
f_{n_{k},4}(x,y,z)=(x+yi)^{n_{k}-2}(x^{2}+y^{2}-a_{k}z^{2})(x^{2}+y^{2}-a_{k+1}z^{2}),
\end{eqnarray*}
where $a_{k}$ and $n_{k}$ are as above and  $k>1$.  
\end{theorem}
\begin{corollary}
The consecutive polynomials $f_{n_{k},4}$ and $f_{n_{k+1},4}$ have a common quadratic 
factor.\qed
\end{corollary}
\begin{proof}[Proof of the theorem]
Setting
\begin{eqnarray}\label{changenm}
\begin{array}{c}
u=4n+1,\\ v=2m,
\end{array} 
\end{eqnarray}
in (\ref{eqnm}) we get the Pell equation 
\begin{eqnarray}\label{equv}
u^{2}-6v^{2}=1. 
\end{eqnarray}
Its fundamental solution $\phi_{1}$ is equal to $(5,2)$. We get the recurrence relations  
\begin{eqnarray*}
\begin{cases}
u_{k+1}=5u_{k}+12v_{k},\\ 
v_{k+1}=2u_{k}+5v_{k} 
\end{cases}
\end{eqnarray*}
for solutions to (\ref{equv}) and (\ref{recrel6}) for the linear combinations of $u_{k}$ and $v_{k}$.  
In particular, $\phi_{2}=(49,20)$. This provides the initial data for (\ref{recrel6}). 
It follows from (\ref{recrel6})  that 
\begin{eqnarray}\label{vnmu4}
\begin{array}{rcl}
n_{k}&=&\frac14({u_{k}-1}),\\ 
m_{k}&=&\frac12{v_{k}}
\end{array}
\end{eqnarray}
are integer. Hence the formulas (\ref{changenm}) define one-to-one correspondence between 
$\frS^{+}$ and the set of solutions to (\ref{eqnm}) in positive integers. By (\ref{ABnm4}), 
\begin{eqnarray}\label{nmAB4}
\begin{array}{rcl}
A_{n_{k}}&=&2(n_{k}-m_{k}),\\
B_{n_{k}}&=&2(n_{k}+m_{k}) 
\end{array}
\end{eqnarray}
and the polynomials $f_{n_{k},4}$ are harmonic. Thus we have to prove that 
\begin{eqnarray}
\begin{array}{rcl}\label{aAbB4}
a_{k}&=&A_{n_{k}},\\
a_{k+1}&=&B_{n_{k}}. 
\end{array}
\end{eqnarray}

The initial data for $u_{k},v_{k}$ are $u_{0}=1$, $u_{1}=5$ and $v_{0}=0$, $v_{1}=2$. 
Hence  $\al_{0}=u_{0}-2v_{0}$ and $\al_{1}=u_{1}-2v_{1}$.  Due to (\ref{recrel6}) this implies  
\begin{eqnarray*}
\al_{k}=u_{k}-2v_{k}. 
\end{eqnarray*}
for all $k$. 
The sequence $\al_{k+1}$ also satisfies  (\ref{recrel6}) and the initial conditions $\al_{1}=u_{0}+2v_{0}=1$ ,
$\al_{2}=u_{1}+2v_{1}=9$. Therefore, 
\begin{eqnarray*}
\al_{k+1}=u_{k}+2v_{k}. 
\end{eqnarray*}
Hence $\al_{k}+\al_{k+1}=2u_{k}$ and $\al_{k+1}-\al_{k}=4v_{k}$. By (\ref{vnmu4}), 
\begin{eqnarray*}
&n_{k}=\frac18(\al_{k}+\al_{k+1}-2),\\ 
&m_{k}=\frac18(\al_{k+1}-\al_{k}). 
\end{eqnarray*}
It follows that 
\begin{eqnarray*}
&A_{n_{k}}=2(n_{k}-m_{k})=\frac12(\al_{k}-1),\\ 
&B_{n_{k}}=2(n_{k}+m_{k})=\frac12(\al_{k+1}-1). 
\end{eqnarray*}
This proves (\ref{aAbB4}).
\end{proof}
\subsection{Table and formulas}
Here are the first six terms of the above quantities. We drop the cases $k=0$ and $k=1$ since they 
do not relate to harmonic polynomials of the degree $n_{k}+2$.   
\begin{center}\small
\begin{tabular}{cccccccc}
$k$&2&3&4&5&6&7\\
$u_{k}$&49&485&4801&47525&470449&4656965 \\
$v_{k.}$&20&198&1960&19402&192060&1901198\\
$n_{k}$&12&121&1200&11881&117612&1164241 \\
$m_{k.}$&10&99&980&9701&96030&950599\\
$a_{k}$&4&44&440&4360&43164&427284 \
\end{tabular}
\end{center}  
The matrix $M$ has eigenvalues $5\pm2\sqrt6$. Here are the explicit formulas for $u_{k}$ and $v_{k}$:
\begin{eqnarray*}
\begin{array}{rcl}
u_{k}&=&\frac{(5+2\sqrt6)^{k}+(5-2\sqrt6)^{k}}{2},\\
v_{k}&=&\frac{(5+2\sqrt6)^{k}-(5-2\sqrt6)^{k}}{2\sqrt6}.
\end{array}
\end{eqnarray*}
The formulas for their generating functions follow: 
\begin{eqnarray*}
\begin{array}{c}
G_{u}(t)=\frac{1-5t}{t^{2}-10t+1},\\ 
G_{v}(t)=\frac{2t}{t^{2}-10t+1}.
\end{array}
\end{eqnarray*}
For $\al_{k}$ and $m_{k}$ we have the functions $G_{u}-2G_{v}$ and $\frac12 G_{v}$, respectively.
Due to  (\ref{vnmu4}) and (\ref{aknk4}), 
\begin{eqnarray*}
&G_{n}(t)=\frac14(G_{u}(t)-\frac1{1-t})=\frac{t(1+t)}{(1-t)(1-10t+t^{2})},\\
&G_{a}(t)=\frac12(G_{u}(t)-2G_{v}(t)-\frac1{1-t})=\frac{4t^{2}}{(1-t)(1-10t+t^{2})}.
\end{eqnarray*}
\section{Reducibility of $p_{n,5}$ and a Pell type equation}
In this section we consider the harmonic polynomials of the type 
\begin{eqnarray}\label{poly5}
f_{n,5}
=\ze^{n-2}z (\ze\bar\ze - A_{n}z^2) (\ze\bar\ze - B_{n}z^2).
\end{eqnarray}
with rational $A_{n},B_{n}$. The same method works with some complications. 
\subsection{The Diophantine equation} 
It is convenient to formulate precisely the reduction to the Diophantine equation. 
\begin{lemma}\label{ABharm}
Suppose $A_{n},B_{n}\in\bbQ$. The polynomial $f_{n,5}$ is harmonic of degree $n+3$ if and only if 
\begin{eqnarray}\label{ABnm}
\begin{cases}
A_{n}=\frac{2}{3}(n-m),\\
B_{n}=\frac{2}{3}(n+m),
\end{cases}
\end{eqnarray}
where $n,m\in\bbN\setminus\{1\}$ satisfy the Diophantine equation 
\begin{eqnarray}\label{2n3n5m}
3 n + 2 n^2=5m^{2}.
\end{eqnarray} 
\end{lemma}
\begin{proof}
Computing $\De f_{n,5}$, we get  
\begin{eqnarray*}
\begin{cases}
3(A_{n}+B_{n})=4n,\\
5 A_{n}B_{n}=(n-1)(A_{n} + B_{n}).
\end{cases}
\end{eqnarray*}
Assuming $A_{n}\leq B_{n}$, we write the solution as 
\begin{eqnarray*}
\begin{array}{c}
A_{n}= \frac23\left(n - \sqrt{\frac{3 n + 2 n^{2}}{5}}\right),\\
B_{n}= \frac23\left(n + \sqrt{\frac{3 n + 2 n^2}{5}}\right).
\end{array}
\end{eqnarray*}
We assume $n\geq2$ since $f_{0,5}=\bar\ze^{2}z$ and $f_{1,5}=\bar\ze z(\ze\bar\ze-\frac43 z^{2})$. 
Set 
\begin{eqnarray*}
m=\sqrt{\frac{3 n + 2 n^2}{5}}. 
\end{eqnarray*}
Clearly, $m>1$ and $m$ can be either integer or irrational. The assumption that $A_{n}$ or $B_{n}$ is rational  
implies that $m\in\bbZ$ and $n,m$ satisfies (\ref{2n3n5m}). 
\end{proof}
\subsection{The Pell type equation}
For $u,v$ defined as 
\begin{eqnarray}
\begin{cases}\label{uvtonm}
u=4n+3,\\ 
v=2m,  
\end{cases}
\end{eqnarray}
the equation (\ref{2n3n5m}) implies the equality  
\begin{eqnarray}\label{pell9}
u^2 - 10 v^2 = 9. 
\end{eqnarray}
Every solution to the Pell equation 
\begin{eqnarray}\label{pell1}
u^{2}-10v^{2}=1
\end{eqnarray}
after multiplication by 3 satisfies (\ref{pell9}). The fundamental solution to (\ref{pell1}) is $(19,6)$. 
Therefore,  
\begin{eqnarray*}
M=\left(\begin{array}{cc}19&60\\6&19\end{array}\right).
\end{eqnarray*}
We consider the solutions $(u,v)$ to (\ref{pell9}) with $u>0$ and use the notation of 
subsection~\ref{pellintro} of Introduction. Set 
\begin{eqnarray*}
&T_{0}=
\{(7,-2),\,(3,0),\,(7,2)\}.
\end{eqnarray*}
\begin{lemma}\label{threeser}
The family $\frS$ of the solutions to the Diophantine  equation (\ref{pell9}) lying in the right halfplane is the 
union of the orbits of the vectors of the triple $T_{0}$ under the action of $\Ga$. 
\end{lemma}
\begin{proof}
The hyperbola $u^{2}-10u^{2}=9$ contains the family of solutions to (\ref{pell9}). 
It is clear that its branch $B$ lying in the right  halfplane contains the above mentioned vectors, the group $\Ga$ 
generated by $M$ preserves $B$ and the lattice $\bbZ^{2}$ and consequently their orbits. Thus we have to 
prove that the orbits are disjoint and their union exhaust the family of solutions. 

The equality $M(7,-2)^{\tsfT}=(13,4)^{\tsfT}$ shows that 
the arc in $B$ with endpoints $(7,-2)$ and $(13,4)$ is the fundamental domain for the action of $\Ga$ on $B$
which is equivalent to the shift $t\to t+1$ in $\bbR$. The straightforward computation shows that 
$9+10v^{2}$ for $v=-2,\dots,3$ is the square of some $u\in\bbN$ only in the cases $(u,v)\in T_{0}$. 
This proves the lemma. 
\end{proof}
In the next lemma we characterize the solutions relating to (\ref{2n3n5m}) among all solutions to (\ref{pell9}). 
Let $\frS^{+}$ denote the family of solutions to (\ref{pell9}) which correspond to that of (\ref{2n3n5m}) 
via (\ref{uvtonm}) with integer $n\geq2$ and set 
\begin{eqnarray*}
&T_{k}=M^{2k}T_{0}.
\end{eqnarray*}
We have 
\begin{eqnarray*}
&M^{2}=\left(\begin{array}{cc}721&2280\\228&721\end{array}\right),\\
&T_{1}=\{(487, 154),~ (2163, 684),~(9607, 3038)\}.
\end{eqnarray*}
\begin{lemma}\label{nmvsuv}
The set $\frS^{+}$ is the union of the triples $T_{k}$, $k\in\bbN$. 
\end{lemma}
\begin{proof}
We note some properties on the action of $\Ga$.  The recurrence relation defined by $M$ implies
\begin{eqnarray*}
\begin{cases}
u_{k+1}=-u_{k},\\
v_{k+1}=2u_{k}-v_{k}
\end{cases}
\mod 4.
\end{eqnarray*}
It is obvious that 
\begin{itemize} 
\item [$\circ$] if $u_{1}$ is odd and $v_{1}$ is even, then the numbers $u_{k}$ of the type $4n+3$ 
and $4n+1$ alternate 
and  $v_{k}$ remains even, 
\item[$\circ$] the first components of vectors in $T_{0}$ are of the type $4n+3$
and the second ones are even, 
\item[$\circ$] $T_{0}\cap\frS^{+}=\varnothing$ since $n\geq2$ implies $4n+3\geq11$.  
\end{itemize}
Using Lemma~\ref{ABharm}, it is easy to check the inclusion $T_{1}\subseteq\frS^{+}$. Positivity of the
entries of $M^{2}$ implies  $T_{k}\subseteq\frS^{+}$ for all $k\in\bbN$. The inverse inclusion follows from 
Lemma~\ref{threeser}. The remaining assertion is evident. 
\end{proof}
\subsection{Description of $\frR_{5}$} 
In the theorem below, we use the indices $k,\eps$ for the entries of $\frS^{+}$, where $k$ stands for the number 
of the triple and $\eps$ indicates the orbit of $\Ga$. The index $\eps$ takes values $-1,0,1$ in the order of 
the representatives of the orbits in the triple $T_{0}$. For example, $v_{k,1}$ denotes the second component 
of the third vector in the triple $T_{k}$.  Note that the indices agree with the natural order on both components. 
\begin{theorem}\label{descr5}
Let the polynomial $f_{n,5}$ defined by (\ref{poly5}) be harmonic and let $A_{n}$ or $B_{n}$ be rational. Then 
there are  $k\in\bbN$ and $\eps\in\{-1,0,1\}$ such that 
\begin{eqnarray}\label{nks34}
&n=\frac14(u_{k,\eps}-3)
\end{eqnarray}
and, if $A_{n}\leq B_{n}$, 
\begin{eqnarray}\label{ABnks}
\begin{array}{c}
A_{n}=\frac16(u_{k,\eps}-2v_{k,\eps}-3),\\[4pt]
B_{n}=\frac16(u_{k,\eps}+2v_{k,\eps}-3), 
\end{array}
\end{eqnarray} 
Conversely, the above formulas 
and (\ref{poly5}) define a harmonic polynomial  for every $k\in\bbN$ and 
$\eps\in\{-1,0,1\}$.  
\end{theorem}
\begin{proof}
The family $\frR_{5}$ is described implicitly in Lemma~\ref{nmvsuv}. Thus we have to assign to the indices 
$k,\eps$ the numbers $n$, $A_{n}$ and $B_{n}$. The formulas (\ref{nks34}) and (\ref{ABnks}) do it. 
The formula (\ref{nks34}) and the equality $m=\frac12v_{k,\eps}$ follow from (\ref{uvtonm}). Substituting this 
to (\ref{ABnm}) we get  (\ref{ABnks}). It follows from Lemma~\ref{threeser}, Lemma~\ref{ABharm}, and 
Lemma~\ref{nmvsuv}, that the assumptions of the theorem hold if and only if $f_{n,5}$ is harmonic. 
\end{proof}
The table below contains the first four entries of the three series of the numbers $n=n_{k,\eps}$ defined by (\ref{nks34}). The indices $k=2,..,5$ and $\eps=-1,0,1$ enumerate columns and rows, respectively.  
\begin{eqnarray*}
\begin{array}{cccc}
121&175561&253159921&365056431601\\ 
540&779760&1124414460&1621404872640\\ 
2401&3463321&4994107561&7201499640721
\end{array}
\end{eqnarray*}

Up to the end of this section we change the notation for the polynomials $p_{n,5}\in\frR_{5}$ and the 
numbers $A_{n},B_{n}$ replacing $n$ with the corresponding $k,\eps$  in accordance with Theorem~\ref{descr5} and dropping $5$. Thus 
\begin{eqnarray*}
p_{k,\eps}=p_{n_{k,\eps},5}. 
\end{eqnarray*}
The matrix $M^{2}$ defines the following recurrence relation for the sequences $u_{k,s},v_{k,s}$:  
\begin{eqnarray}\label{rec721}
X_{k+1}=1442X_{k}-X_{k-1}.
\end{eqnarray}
The initial data for all series are contained in the triples $T_{0}$ and $T_{1}$ that are written 
above Lemma~\ref{threeser} and Lemma~\ref{nmvsuv}, respectively. 

Every $k$ correspond six numbers $A_{k,\eps},B_{k,\eps}$. In fact, there are only four 
distinct of them by the following theorem. 
\begin{theorem}\label{fourAB}
For any $k\in\bbN$ 
\begin{itemize}
\item[{\rm(\romannumeral1)}]  $A_{k,0}=B_{k,-1}$ and $B_{k,0}=A_{k,1}$,
\item [{\rm(\romannumeral2)}] the above numbers are integer, the remaining $A_{k,-1}$ and $B_{k,1}$ are not integer. 
\end{itemize} 
\end{theorem}
\begin{proof} 
Due to (\ref{ABnks}),  \rm(\romannumeral1) is equivalent to the equalities 
\begin{eqnarray*}
&u_{k,-1}+2v_{k,-1}=u_{k,0}-2v_{k,0},\\
&u_{k,0}+2v_{k,0}=u_{k,1}-2v_{k,1}. 
\end{eqnarray*}
Both left and right parts of these equalities  are subject to (\ref{rec721}). Hence it is sufficient to check them  
for $k=0$ and $k=1$. This is an easy calculation which proves (\romannumeral1).  

According to (\ref{ABnks}),  (\ref{uvtonm}), and (\ref{ABnm}) the recurrence relations 
for $u_{k,\eps},v_{k,\eps}$ defined by $M^{2}$ may be rewritten for $A_{k,\eps}, B_{k,\eps}$  as follows:
\begin{eqnarray}\label{recAB}
\begin{cases}
A_{k+1,\eps}=-77A_{k,\eps}+342B_{k,\eps}+132,\\
B_{k+1,\eps}=-342A_{k,\eps}+1519B_{k,\eps}+588.
\end{cases}
\end{eqnarray}
By (\ref{ABnks}), the initial data relating to $T_{0}$  are $(\frac43,0,0)$ for $A_{0,\eps}$ and
$(0,0,\frac43)$ for $B_{0,\eps}$, where $\eps=-1,0,1$, respectively. Since 
\begin{eqnarray*}
-77=1519=1\mod3
\end{eqnarray*}
and other coefficients on the right are divisible by $3$, the fractional parts of $A_{k,-1}$ and 
$B_{k,1}$ equals $\frac13$ for all $k\in\bbN$ while the two others $A_{k,\eps},B_{k,\eps}$  are integer. 
Thus (\romannumeral2) is true. 
\end{proof}
\begin{corollary}
In every triple  $T_{k}=(p_{k.-1},p_{k,0},p_{k,1})$  the consecutive polynomials have a common quadratic 
factor. Polynomials from different triples and $p_{k,-1},p_{k,1}$ have no common factor. \qed
\end{corollary}
\begin{corollary}
A quadratic polynomial   $x^{2}+y^{2}-Az^{2}$ divides at most two  distinct $p_{k,\eps}$. This happens 
if and only if it divides one of them and $A$ is integer. \qed 
\end{corollary}
\subsection{Tables and formulas} According to (\ref{uvtonm}), for $n_{k,\eps},m_{k,\eps}$  we have 
\begin{eqnarray}\label{recnm}
\begin{cases}
n_{k+1,\eps}=721n_{k,\eps}+1140m_{k,\eps}+540,\\
m_{k+1,\eps}=456n_{k,\eps}+721m_{k,\eps}+342.
\end{cases}
\end{eqnarray}
The initial data relating to $T_{0}$ are 
\begin{eqnarray}\label{indall}
\begin{array}{cccc}
(u_{0,\eps},v_{0,\eps})&(7,-2),&(3,0),&(7,2),\\[2pt]
(n_{0,\eps},m_{0,\eps})&(1,-1), &(0,0)&(1, 1),\\[2pt]
(A_{0,\eps},B_{0,\eps})&\big(\frac43,0\big), &\big(0,0\big)&\big(0, \frac43\big).
\end{array}
\end{eqnarray}

The linear parts of the affine transformations on the right of (\ref{recnm})  and (\ref{recAB}) are conjugate to $M^{2}$.  
Hence we get the sequences subject to (\ref{rec721}) removing the origin to the fixed points of these transformations 
which are equal to $(-\frac34,0)$ and $(-\frac12,-\frac12)$, respectively. Let $X_{k}(x_{0},x_{1})$ be the sequence 
satisfying (\ref{rec721}) with the initial data $X_{0}(x_{0},x_{1})=x_{0}$, $X_{1}(x_{0},x_{1})=x_{1}$. Then 
\begin{eqnarray}\label{recns}
\begin{array}{rcl}
n_{k,\eps}&=&X_{k}(n_{0,\eps}+\frac34,n_{1,\eps}+\frac34)-\frac34,\\[3pt]
A_{k,\eps}&=&X_{k}(A_{0,\eps}+\frac12,A_{1,\eps}+\frac12)-\frac12,
\end{array}
\end{eqnarray} 
$m_{k,\eps}=X_{k}(m_{0,\eps},m_{1\eps})$, and $B_{k,\eps}$ is subject to the same formula as $A_{k,\eps}$. 
The data relating to $k=0$ are given in (\ref{indall}). Here is the similar table for $k=1$: 
\begin{eqnarray}\label{indall1}
\begin{array}{cccc}
(u_{1,\eps},v_{1,\eps})&(487, 154),& (2163, 684),&(9607, 3038)\\[2pt]
(n_{1,\eps},m_{1,\eps})&(121,77), &(540,342)&(2401, 1519),\\[2pt]
(A_{1,\eps},B_{1,\eps})&\big(\frac{88}3,132\big), &\big(132,588\big)&\big(588, \frac{7840}3\big).
\end{array}
\end{eqnarray}
For the generating function $G_{X}(t)=\sum_{k=0}^{\infty}X_{k}t^{k}$, where $X_{k}=X_{k}(x_{0},x_{1})$ is as above, 
we have 
\begin{eqnarray*}
&G_{X}(t)=\frac{x_{0}-(1442x_{0}-x_{1})t}{1-1442t+t^{2}}.
\end{eqnarray*}
Combining it with (\ref{recns}) and using the initial data, we get the formulas for all series:
\begin{eqnarray*}
\begin{array}{cccc}
G_ {n,\eps} &\frac {121 + 958 t + t^2} {(1 - t) (1 - 1442 t + t^2)} 
&\frac {540 (1 + t)} {(1 - t) (1 - 1442 t + t^2)} &\frac {2401 + 1322 t + t^2} {(1 - t) (1 - 1442 t + t^2)} \\[6 pt]
G_ {m,\eps} &\frac {77 + t} {1 - 1442 t + t^2} &\frac {342} {1 - 1442 t + t^2} &\frac {1519 - t} {1 - 1442 t + t^2}
\end{array}  
\end{eqnarray*}
The generating functions for $A_{k,\eps},B_{k,\eps}$ also may be calculated in this way or  
by the formulas $G_{A,\eps}=\frac23(G_{n,\eps}-G_{m,\eps})$, $G_{B,\eps}=\frac23(G_{n,\eps}+G_{m,\eps})$ which follow 
from (\ref{ABnm}):
\begin{eqnarray*}
\begin{array}{cccc}
G_ {A,\eps} &\frac{4 (22 + 517 t + t^2)}{3 (1 - t) (1 - 1442 t + t^2)}&\frac{12 (11 + 49 t)}{(1 - t) (1 - 1442 t + t^2)} 
&\frac{12 (49 + 11 t)}{(1 - t) (1 - 1442 t + t^2)} \\[6 pt]
G_ {B,\eps} &\frac{12 (11 + 49 t)}{(1 - t) (1 - 1442 t + t^2)} &\frac{12 (49 + 11 t)}{(1 - t) (1 - 1442 t + t^2)} 
&\frac{4 (1960 - 1421 t + t^2)}{3 (1 - t) (1 - 1442 t + t^2)}
\end{array}  
\end{eqnarray*}
The sequences $A_{k,\eps}$ satisfy the inequalities 
\begin{eqnarray*}
A_{1,-1}<A_{1,0}<A_{1,1}<A_{2,-1}<A_{2,0}<A_{2,1}<\dots.
\end{eqnarray*}
The generating function for the sequence of all $A_{n}$ enumerated in this order is 
$G_{A}=G_{A,-1}+t\,G_{A,0}+t^{2}\,G_{A,1}$. The function $G_{B}$ can be defined analogously. 
They are too awkward to work with but a computer computation shows that the following equality holds: 
\begin{eqnarray*}
G_{A}(t)-t\,G_{B}(t)=\frac{4}{3}+\frac{28 (1+t^3)}{1 - 1442 t^3 + t^6}
\end{eqnarray*}
It gives a hint of another proof of Theorem~\ref{fourAB}. The right-hand part characterizes the gaps 
between the triples. 
\section{The sets $\frR_{6}$ and $\frR_{7}$ are finite}
\subsection{The case $d=6$}  
As in the cases 
$k=4$ and $k=5$, we may write $f_{n,6}$ as 
\begin{eqnarray*}
\ze^{n-3}(\ze\bar\ze-A_{n}z^{2})(\ze\bar\ze-B_{n}z^{2})(\ze\bar\ze-C_{n}z^{2}).
\end{eqnarray*} 
If $f_{n,6}$ is reducible, then one of the numbers $A_{n},B_{n},C_{n}$ is rational. 
The equation  $\De f_{n,6}=0$ is equivalent to a system of linear equations with the elementary 
symmetric polynomials  of $A_{n},B_{n},C_{n}$. Resolving it we get the polynomial 
$p_{n,6}(x,y,z)=z^{6}\td p_{n,6}\left(\frac{x^{2}+y^{2}}{z^{2}}\right)$, where 
\begin{eqnarray*}
\td p_{n,6}(t)=t^{3}-6nt^{2}+4n(n-1)t-\frac{8}{15}n(n-1)(n-2)
\end{eqnarray*}
with roots $A_{n},B_{n},C_{n}$. It is more convenient to work with  the polynomial 
$q_{n,6}(t)=\frac{1}{8}\,\td p_{n,6}\left(2t+2n\right)$ that is the canonical form of $\td p_{n,6}$: 
\begin{eqnarray}\label{deftdp}
q_{n,6}(t)=t^{3}-n(2n+1)t-\frac{2}{15} n(2n+1)(4n+1). 
\end{eqnarray}   
Let $p,r\in\bbN$ and $p$ be prime. Recall that $p|m$ indicates that $p$ divides $m$. Set 
\begin{eqnarray*}
h(p,r)=\max\{\al\in\bbN\cup\{0\}:p^{\al} |r\}.  
\end{eqnarray*}
The definition can be extended onto the rational numbers $r=\frac{m}{k}$ by 
\begin{eqnarray*}
h(p,r)=h(p,m)-h(p,k).  
\end{eqnarray*}
Then $p|r$ means that $h(p,r)>0$. 
 \begin{lemma} \label{redTSR6}
 Suppose that $q_{n,6}$ admits a rational root. Let $p$ be a prime divisor of $n(2n+1)$. 
If  $p\neq2,3,5$, then $h(p,n(2n+1))$ is divisible by $3$.  
\end{lemma}
\begin{proof}
Let $r=\frac{m}{k}$, where $m,k$ are coprime integers. Recall that if $r$ is a root of a polynomial 
$a_{n}t^{n}+\dots+a_{0}$ with integer coefficients, then $m|a_{0}$ and $k|a_{n}$. 


The second and the third coefficients of the polynomial $3\cdot5^{3}\,q_{n,6}\left(\frac{k}{5}\right)$ are integer 
and $n(2n+1)$ is their common factor. The assumption $p\neq2,3,5$ implies 
\begin{eqnarray*}
h(p,15n(2n+1))=h(p,2n(2n+1)(4n+1))
\end{eqnarray*}
since $n$, $2n+1$, and $4n+1$ are pairwise coprime. Let us denote by $\al$ the above number and 
set ${\be}=h(p,c)$.  We have $\al>0$ and consequently ${\be}=h(p,c)=\frac13 h(p,c^{3})>0$.  Put 
\begin{eqnarray}\label{albebal}
\mu=\min\{3\be,\al+\be,\al\}. 
\end{eqnarray}
Dividing (\ref{deftdp}) with $t=c$ by $p^{\mu}$ we get a sum of three rational numbers such that at least one 
of them is not  divisible by $p$. We get a contradiction if the others are divisible because the sum equals zero. 
Hence there are at least two minimal numbers in the triple $3\be, \al+\be,\al$. Then the inequality $\be>0$ 
implies $3\be=\al$. This proves the lemma. 
\end{proof}
\subsection{The case $d=7$} We perform a preparatory work for the case $d=7$ before proving the main result. 
It is similar to that is above.  
We have 
\begin{eqnarray*}
\td p_{n,7}(t)= t^{3}-2nt^{2}+\frac45n(n-1)t-\frac{8}{105}n(n-1)(n-2)
\end{eqnarray*}
and the following expression for $q_{n,7}(t)=\td p_{n,7}(t+\frac23 n)$:
\begin{eqnarray*}
q_{n,7}(t)=t^{3}-\frac4{15}n(2n+3)t-\frac{16}{945}n(2n+3)(4n+3).
\end{eqnarray*}
 \begin{lemma} \label{redTSR7}
 Let $q_{n,7}$ have a rational root and let $p$ be a prime divisor of $n(2n+3)$. 
If  $p\neq2,3,5,7$, then $h(p,n(2n+3))$ is divisible by $3$.  
\end{lemma}
\begin{proof}
We consider the polynomial $945q_{n.7}$ which has integer coefficients. Note that $945=9!!=3^{3}\cdot5\cdot7$. 
The assumption implies 
\begin{eqnarray*}
h(p,2^{2}\cdot3^{2}\cdot7\,n(2n+3))=h(p,16n(2n+3)(4n+3))=:\al>0
\end{eqnarray*}
since the greatest common divisor of every pair in the triple $n$, $2n+3$, and $4n+3$ is either $3$ or $1$. 
Set ${\be}=h(p,c)$ and let $\mu$ be defined by (\ref{albebal}). Clearly, $\be>0$.  
As in the case $d=6$, there are at least two minimal numbers in the triple $3\be, \al+\be,\al$. Hence 
$\al=3\be$.  This proves the lemma. 
\end{proof}

\subsection{Almost all polynomials $p_{n,6}$ and $p_{n,7}$ are irreducible}
 \begin{theorem}
 The sets $\frR_{6}$ and $\frR_{7}$ are finite.
 \end{theorem}
 \begin{proof}
 Since $\deg p_{n,6}=\deg p_{n,7}=3$, the reducibility of these polynomials is equivalent to 
 existence of a rational root. Thus we may apply Lemma~\ref{redTSR6} and Lemma~\ref{redTSR7}, 
 respectively. Let $d=6$. Then  
\begin{eqnarray*}
n(2n+1)=2^{\al_{1}}3^{\al_{2}}5^{\al_{3}}p_{1}^{3\be_{1}}p_{2}^{3\be_{2}}\cdots p_{k}^{3\be_{k}}, 
\end{eqnarray*}
where $p_{1},\dots,p_{k}$ are primes distinct from $2,3,5$.  Similar factorization holds for both $n$ and $2n+1$ 
since they are coprime. It follows that 
\begin{eqnarray*}
n=Ku^{3}~~\text{and}~~ 2n+1=Lv^{3}, 
\end{eqnarray*}
where $u,v\in\bbN$ and $K,L=2^{\al}3^{\be}5^{\ga}$ with some $\al,\be,\ga\in\{0,1,2\}$.  
Hence $u,v$ satisfy the Diophantine equation 
\begin{eqnarray}\label{kluvdio}
Lv^{3}-2Ku^{3}=1. 
\end{eqnarray}
It is well known that the set of solutions to every such equation is finite due to the celebrated  Thue--Siegel--Roth 
theorem:  
\begin{itemize}
\item {\it for any irrational algebraic number $a$ and $\ep>0$ there exists $C>0$ such that
the inequality $\left|a-\frac{l}{k}\right|>\frac{C}{k^{2+\ep}}$ holds for every $l\in\bbZ$ and $k\in\bbN$}. 
\end{itemize}
It obviously implies that for any $C>0$ the inverse inequality $\left|a-\frac{l}{k}\right|<\frac{C}{k^{2+\ep}}$
may be true only for finite set of $\frac{l}{k}$. For solutions $u,v$ to (\ref{kluvdio}) it is easy to derive the 
inequalities $0<a-\frac{u}{v}<\frac{1}{2Ka^{2}v^{3}}$, where $a=\sqrt[3]{\frac{L}{2K}}$. Thus the 
set $\frR_{6}$ is finite. 

The above arguments with minor changes prove that $\frR_{7}$ is finite. We have to add $7$ 
to the set $\{2,3,5\}$, note that only $3$ may be a nontrivial common divisor of $n$ 
and $2n+3$, replace  (\ref{kluvdio}) with the equation $Lv^{3}-2Ku^{3}=3$, and multiply the upper 
bound for $a-\frac{u}{v}$ by $3$. This proves the theorem. 
 \end{proof}
\section{Remarks} 
\subsection{} There are more than a hundred Diophantine equations that satisfy the above conditions. 
Perhaps, a half or more of them have no solution but this needs a careful analysis.  The books \cite{DF40}, \cite{Mo69} 
contain fundamental facts concerting Diophantine equations of third degree.  For example,  it is known that the 
equation $u^{3}-Dv^{3}=1$, where $D$ is free of cubes, may have at most one solution distinct from the trivial $u=1$, 
$v=0$ (see \cite[Ch. VI, \S71, Theorem V]{DF40} or \cite[Ch. 24, Theorem~5]{Mo69}).  Note that the existence of a 
solution does not imply the reducibility of the relating polynomial $p_{n,6}$ or $p_{n,7}$.
\subsection{} Let $C$ be a round cone in $\bbR^{3}$ which is not a plane and $\cH(C)$ be the space of all
harmonic polynomials that vanish on $C$. There are such cones $C$ in $\bbR^{4}$ with infinite dimensional 
$\cH(C)$ (see \cite[Example~4.3]{LM15} ). To the best of my knowledge, no such example in $\bbR^{3}$ is known. 
The functions $f_{n,d}$, $d\leq5$, give examples of $C$ with $\dim\cH(C)\geq8$. 
For example,  
\begin{itemize} 
\item $x^{2}+y^{2}-44z^{2}$ divides $p_{22,2}$, $p_{66,3}$, $p_{10,4}$, and $p_{99,4}$,
\item $x^{2}+y^{2}-132 z^{2}$ divides $p_{66,2}$, $p_{198,3}$, $p_{77,5}$, and $p_{540,5}$.
\end{itemize}
Every $p_{n,d}$ relates to a couple of harmonic functions $f_{n,d}$ and $\bar f_{n,d}$ with the same zeros. 
I was not able to find common quadratic divisors for $p_{n,4}$ and $p_{m,5}$.


\end{document}